\documentclass{article}
\usepackage{amsfonts,latexsym,hyperref} % label
\newcounter{satznum}
\newtheorem{theorem}{Theorem}[satznum]

\newtheorem{lemma}[theorem]{Lemma}
\newtheorem{corollary}[theorem]{Corollary}

\newenvironment{remark}
 {\begin{trivlist}\item[]{\bf Remark.}}
 {\end{trivlist}}
\newenvironment{remarks}
 {\begin{trivlist}\item[]{\bf Remarks.}}
 {\end{trivlist}}

\newenvironment{proof}
 {\begin{trivlist}\item[]{\bf Proof.}}
 {\end{trivlist}}
{\makeatletter
\gdef\me{{\mathbb E}} % expectation
\gdef\nz{{\mathbb N}} % positive integers
\gdef\pr{{\mathbb P}} % probability
\gdef\rz{{\mathbb R}} % real numbers
}
\newcounter{todocounter}

\makeatletter
\def\@MRExtract#1 #2!{#1}
\newcommand{\MR}[1]{% we need to strip the "(...)"
  \xdef\@MRSTRIP{\@MRExtract#1 !}
  \href{http://www.ams.org/mathscinet-getitem?mr=\@MRSTRIP}{MR\@MRSTRIP}}
\makeatother

\begin{document}
   \section*{THE MITTAG--LEFFLER PROCESS AND A SCALING LIMIT FOR THE BLOCK COUNTING PROCESS OF THE
   BOLTHAUSEN--SZNITMAN COALESCENT}
   {\sc M. M\"ohle}\footnote{Mathematisches Institut, Eberhard Karls Universit\"at T\"ubingen,
   Auf der Morgenstelle 10, 72076 T\"ubingen, Germany, E-mail address: martin.moehle@uni-tuebingen.de}
\begin{center}
%   Draft Version\\
%   \today
   October 23, 2014
\end{center}
\begin{abstract}
   The Mittag--Leffler process $X=(X_t)_{t\ge 0}$ is introduced.
   This Markov process has the property that its marginal random
   variables $X_t$ are Mittag--Leffler distributed with parameter
   $e^{-t}$, $t\in [0,\infty)$, and the semigroup $(T_t)_{t\ge 0}$ of $X$
   satisfies $T_tf(x)=\me(f(x^{e^{-t}}X_t))$ for all $x\ge 0$ and all
   bounded measurable functions $f:[0,\infty)\to\rz$.
   Further characteristics
   of the process $X$ are derived, for example an explicit formula
%   e for its semigroup and
   for the
   joint moments of its finite dimensional distributions. The main result
   states that the block counting process of the Bolthausen--Sznitman
   $n$-coalescent, properly scaled, % by $n^{e^{-t}}$,
   converges in the Skorohod topology to the Mittag--Leffler process
   $X$ as the sample size $n$ tends to infinity.

   \vspace{2mm}

   \noindent Keywords: block counting process; Bolthausen--Sznitman
   coalescent; marginal distributions; Mittag--Leffler process; weak
   convergence

   \vspace{2mm}

   \noindent 2010 Mathematics Subject Classification:
            Primary 60F05;   % Central limit and other weak theorems
            60J27            % Continuous-time Markov processes on discrete
                             % state space
            Secondary 92D15; % Problems related to evolution
            97K60            % Distributions and stochastic processes
\end{abstract}
\subsection{Introduction and main results} \label{intro}
\setcounter{theorem}{0}
Exchangeable coalescent processes with multiple collisions are Markov
processes with state space ${\cal P}$, the set of partitions of
$\nz:=\{1,2,\ldots\}$. During each transition blocks merge together
to form a single block. These processes are characterized by a measure
$\Lambda$ on the unit interval $[0,1]$. For more information on these
processes we refer the reader to \cite{pitman} and \cite{sagitov}.
The Bolthausen--Sznitman coalescent \cite{bolthausensznitman} is the
particular $\Lambda$-coalescent $\Pi=(\Pi_t)_{t\ge 0}$ with
$\Lambda$ being the uniform distribution on $[0,1]$. In this
article we focus on the process $\Pi^{(n)}=(\Pi_t^{(n)})_{t\ge 0}$ of the
Bolthausen--Sznitman coalescent $\Pi$ restricted to a sample of size
$n\in\nz$. We are in particular interested in the process
$N^{(n)}:=(N_t^{(n)})_{t\ge 0}$, where $N_t^{(n)}$ denotes the number
of blocks of $\Pi_t^{(n)}$. The process $N^{(n)}$ is called the
block counting process of the Bolthausen--Sznitman $n$-coalescent.
It is well known that $N^{(n)}$ is a time-homogeneous Markov chain
with generator $Q=(q_{ij})_{1\le i,j\le n}$ having entries
$q_{ij}:=i/((i-j)(i-j+1))$ if $i>j$, $q_{ij}:=1-i$ if $i=j$ and $q_{ij}=0$
if $i<j$. For $n\in\nz$ and $t\in [0,\infty)$ define
\begin{equation} \label{x}
   X_t^{(n)}\ :=\ \frac{N_t^{(n)}}{n^{e^{-t}}}.
\end{equation}
We call $X^{(n)}:=(X_t^{(n)})_{t\ge 0}$ the scaled block counting process
of the Bolthausen--Sznitman $n$-coalescent. The scaling
$n^{e^{-t}}$ in (\ref{x}) is somewhat unusual since it involves not only
the parameter $n$ but also
the time parameter $t$. Clearly, $X^{(n)}$ is a Markov process with
state space $E:=[0,\infty)$, however, since the scaling depends
on $t$, $X^{(n)}$ is time-inhomogeneous. Our main result (Theorem
\ref{main} below) provides a distributional limiting result for $X^{(n)}$ as
the sample size $n$ tends to infinity. The arising limiting Markov process
$X=(X_t)_{t\ge 0}$ we call the Mittag--Leffler process, since the marginal
random variable $X_t$ turns out to be Mittag--Leffler distributed with
parameter $e^{-t}$. Note that the distribution of $X_t$ is uniquely
determined by its entire moments $\me(X_t^m)=m!/\Gamma(1+me^{-t})$,
$m\in\nz_0:=\{0,1,2,\ldots\}$.
For detailed information on the Mittag--Leffler distribution and on the
process $X$ we refer the reader to Section \ref{mittag}, where the existence
of $X$ is established and fundamental properties of this process are derived.
In order to wipe out possible confusion with processes in the literature
having similar names we mention that the process $X$ has nothing in common with
the autoregressive Mittag--Leffler process studied for example by Jayakumar
\cite{jayakumar} and Jayakumar and Pillai \cite{jayakumarpillai}.
These processes are based on the (heavy-tailed) Mittag--Leffler distribution
of the first type (see, for example, \cite{mainardigorenflo} and \cite{pillai}
for some related works), whereas the Mittag--Leffler distributed random
variable $X_t$ is of the second type and has finite moments of all orders.
Let us now present our main convergence result.
\begin{theorem} \label{main}
   For the Bolthausen--Sznitman coalescent the scaled block counting process
   $X^{(n)}=(X_t^{(n)})_{t\ge 0}$ defined via (\ref{x}) converges in
   $D_E[0,\infty)$ as $n\to\infty$ to the Mittag--Leffler process
   $X=(X_t)_{t\ge 0}$ introduced in Section \ref{mittag}.
\end{theorem}
\begin{remarks}
   1. Theorem \ref{main} can be also stated logarithmically as follows.
   The process $(\log N_t^{(n)}-e^{-t}\log n)_{t\ge 0}$ converges in
   $D_\rz[0,\infty)$ to the process $(\log X_t)_{t\ge 0}$ as $n\to\infty$.
   Neither $X$ nor $(\log X_t)_{t\ge 0}$ is a L\'evy process. Note that
   the logarithmic block counting process $(\log N_t^{(n)})_{t\ge 0}$
   plays an important role in the problem of whether a coalescent process
   comes down from infinity, see, for example, Section 4 of \cite{limic}.

   2. Note that $X^{(n)}$ is time-inhomogeneous whereas the limiting process
   $X$ is time-homogeneous. Thus, Theorem \ref{main} in particular states
   that $X^{(n)}$ is asymptotically time-homogeneous.
\end{remarks}
The article is organized as follows.
Section \ref{mittag} is devoted to the Mittag--Leffler process $X$.
We prove the existence of this process and derive fundamental properties
of $X$, among them representations for the semigroup of $X$ (see (\ref{semigroup}))
and an explicit formula for the joint moments (see Lemma \ref{xmoments}) of the
finite-dimensional distributions of $X$. In Section \ref{moment} we
provide some fundamental formulas (see Lemma \ref{mean} and Lemma
\ref{meangeneral}) for certain moments of the
block counting process $N^{(n)}$ of the Bolthausen--Sznitman $n$-coalescent.
These results rely on the
spectral decomposition \cite{moehlepitters} of the
generator of the block counting process. Lemma \ref{mean} in particular
shows that $N_t^{(n)}$ has mean
\begin{equation} \label{firstmoment}
   \me(N_t^{(n)})\ =\ \frac{\Gamma(n+e^{-t})}{\Gamma(n)\Gamma(1+e^{-t})}
   \qquad n\in\nz, t\in [0,\infty).
\end{equation}
For large $n$ the mean (\ref{firstmoment}) is asymptotically equal to
$n^{e^{-t}}(\Gamma(1+e^{-t}))^{-1}=n^{e^{-t}}\me(X_t)$, which indicates
that $n^{e^{-t}}$ is the appropriate scaling in order to obtain a
non-degenerate limiting process for the scaled block counting process
as $n$ tends to infinity. In the final
Section \ref{proof} this argument is made rigorous leading to a proof
of Theorem \ref{main}. First the convergence of the finite-dimensional
distributions is verified and afterwards the convergence in $D_E[0,\infty)$
is established.

We leave it open for future work to establish convergence results
in analogy to Theorem \ref{main} for the block counting process $N^{(n)}$
of more general coalescent processes (that do not come down from infinity),
for example for the $\beta(a,b)$-coalescent with $a\ge 1$ (and $b>0$).
\subsection{The Mittag--Leffler process} \label{mittag}
\setcounter{theorem}{0}
Before we come to the Mittag--Leffler process let us briefly mention
some well known results on the Mittag--Leffler distribution.
Let $\eta=\eta(\alpha)$ be a random variable being Mittag--Leffler
distributed with parameter $\alpha\in [0,1]$. Note that $\eta$ has
moments %\cite[Example on p.~480]{gnedinpitmanyor}
$$
\me(\eta^m)\ =\ \frac{\Gamma(1+m)}{\Gamma(1+m\alpha)},\qquad m\in [0,\infty),
$$
%For $\alpha\in (0,1]$ the random variable $\eta$ has Laplace transform
%$$
%s\mapsto \me(e^{-s\eta})
%\ =\ \sum_{m=0}^\infty \frac{(-s)^m}{m!}\me(\eta^m)
%\ =\ \sum_{m=0}^\infty \frac{(-s)^m}{\Gamma(\alpha m+1)}
%\ =\ E_\alpha(-s),\qquad s\in [0,\infty),
%$$
%where $E_\alpha:\cz\to\cz$, $E_\alpha(z):=\sum z^m/\Gamma(\alpha m+1)$,
%denotes the Mittag--Leffler function. Note that, since $\alpha>0$, the series
%$E_\alpha(z)$ converges for all $z\in\cz$.
and that the entire moments $\me(\eta^m)$, $m\in\nz_0$, uniquely
determine the distribution of $\eta$. Clearly, $\eta$ is standard
exponentially distributed for $\alpha=0$ and $\pr(\eta=1)=1$ for
$\alpha=1$.

If $\alpha_n\to\alpha$, then the moments of
$\eta(\alpha_n)$ converge to those of $\eta(\alpha)$, which
implies the convergence $\eta(\alpha_n)\to\eta(\alpha)$ in distribution
as $n\to\infty$. Thus, the map $\alpha\mapsto\pr_{\eta(\alpha)}$ is a
continuous function from $[0,1]$ to the space ${\cal P}(E)$ of
probability measures on $E:=[0,\infty)$ equipped with the topology
of convergence in distribution.

For $\alpha\in (0,1)$ the Mittag--Leffler distribution can be characterized
in terms of an exponential integral of a particular subordinator as follows.
Let $S=(S_t)_{t\ge 0}$ be a drift-free subordinator with killing rate
$k:=1/\Gamma(1-\alpha)$ and L\'evy measure $\varrho$ having density
\begin{equation} \label{rho}
   u\ \mapsto\ \frac{1}{\Gamma(1-\alpha)}
   \frac{e^{-u/\alpha}}{(1-e^{-u/\alpha})^{\alpha+1}},
   \qquad u\in (0,\infty),
\end{equation}
with respect to Lebesgue measure on $(0,\infty)$. It is readily
checked (see Lemma \ref{laplace} in the appendix) that $S$ has
Laplace exponent
\begin{equation} \label{phi}
   \Phi(x)\ =\ \frac{\Gamma(1+\alpha x)}{\Gamma(1-\alpha+\alpha x)},
   \qquad x\in [0,\infty).
\end{equation}
The distribution of the exponential integral $I:=\int_0^\infty
e^{-S_t}{\rm d}t$ is uniquely determined (see \cite{carmonapetityor})
via its entire moments
$$
\me(I^m)\ =\ \frac{m!}{\Phi(1)\cdots\Phi(m)}
\ =\ m!\prod_{j=1}^m \frac{\Gamma(1+(j-1)\alpha)}{\Gamma(1+j\alpha)}
\ =\ \frac{\Gamma(1+m)}{\Gamma(1+m\alpha)},\qquad m\in\nz.
$$
Thus, $I$ is Mittag--Leffler distributed with parameter $\alpha$.

\subsubsection{Existence of the Mittag--Leffler process}
In this subsection we prove the existence of a particular Markov process
$X=(X_t)_{t\ge 0}$ having sample paths in $D_E[0,\infty)$ such that every
$X_t$ is Mittag--Leffler distributed with parameter $e^{-t}$. Constructing
Markov processes with given marginal distributions has attained some
interest in the literature, mainly in the context of (semi)martingales.
We exemplary refer the reader to \cite{madanyor} and the references
therein. Note however, that the process $X$ we are going to construct
will be neither a supermartingale nor a submartingale.
% We essentially use the
% standard semigroup approach being for example well developed in
% \cite{ethierkurtz}.

For $t\in [0,\infty)$ let $\eta_t$ be a random variable being Mittag--Leffler
distributed with parameter $e^{-t}$. % Set $E:=[0,\infty)$ and
Define $p:[0,\infty)\times E\times {\cal B}(E)$ via
\begin{equation} \label{tf}
   p(t,x,B)\ :=\ \me(1_B(x^{e^{-t}}\eta_t))
   \ =\ \pr(x^{e^{-t}}\eta_t\in B).
\end{equation}
The definition of $p$ is such that
for all $(t,x)\in [0,\infty)\times E$ the random variable $x^{e^{-t}}\eta_t$
has distribution $p(t,x,.)$. In particular, $p(t,x,.)$ has moments
\begin{equation} \label{tfmoments}
   \int_E y^m p(t,x,{\rm d}y)
   \ =\ \me((x^{e^{-t}}\eta_t)^m)
   \ =\ x^{me^{-t}}\frac{\Gamma(1+m)}{\Gamma(1+me^{-t})},\qquad m\in [0,\infty),
\end{equation}
and the entire moments $\int_E y^m p(t,x,{\rm d}y)$, $m\in\nz_0$, uniquely
determine the distribution $p(t,x,.)$. In order to verify the
Chapman--Kolmogorov property
\begin{equation} \label{chap}
   p(s+t,x,B)\ =\ \int_E p(s,y,B)\,p(t,x,{\rm d}y),
   \qquad s,t\in [0,\infty), x\in E, B\in{\cal B}(E),
\end{equation}
fix $s,t\in [0,\infty)$ and $x\in E$. Define
$\mu_1(B):=p(s+t,x,B)$ and $\mu_2(B):=\int_E p(s,y,B)\,p(t,x,{\rm d}y)$
for all $B\in{\cal B}(E)$. Clearly, $\mu_1$ and $\mu_2$ are
probability measures on $E$. By (\ref{tfmoments}), $\mu_1$ has moments
$$
\int_E z^m\,\mu_1({\rm d}z)
\ =\ \int_E z^m p(s+t,x,{\rm d}z)
\ =\ x^{me^{-(s+t)}}\frac{\Gamma(1+m)}{\Gamma(1+me^{-(s+t)})},
\qquad m\in\nz_0,
$$
and these moments uniquely determine $\mu_1$. By Fubini's theorem
and (\ref{tfmoments}), $\mu_2$ has moments
\begin{eqnarray*}
   \int_E z^m\,\mu_2({\rm d}z)
   & = & \int_E z^m \int_E p(s,y,{\rm d}z)\,p(t,x,{\rm d}y)
   \ = \ \int_E \bigg(\int_E z^m p(s,y,{\rm d}z)\bigg)\,p(t,x,{\rm d}y)\\
   & = & \int_E y^{me^{-s}}\frac{\Gamma(1+m)}{\Gamma(1+me^{-s})}\,p(t,x,{\rm d}y)
   \ = \ \frac{\Gamma(1+m)}{\Gamma(1+me^{-s})}
         \int_E y^{me^{-s}} p(t,x,{\rm d}y)\\
   & = & \frac{\Gamma(1+m)}{\Gamma(1+me^{-s})}
         x^{me^{-s}e^{-t}}\frac{\Gamma(1+me^{-s})}{\Gamma(1+me^{-s}e^{-t})}
   \ = \ x^{me^{-(s+t)}}\frac{\Gamma(1+m)}{\Gamma(1+me^{-(s+t)})},
\end{eqnarray*}
and these moments uniquely determine $\mu_2$. Since the moments of
$\mu_1$ and $\mu_2$ coincide, it follows that $\mu_1=\mu_2$ and the
Chapman--Kolmogorov property (\ref{chap}) is established. Thus,
the family $(T_t)_{t\ge 0}$ of linear operators $T_t$, defined via
\begin{equation} \label{semigroup}
   T_tf(x)\ :=\ \int_E f(y)\,p(t,x,{\rm d}y)
   \ =\ \me(f(x^{e^{-t}}\eta_t)),\qquad t\in [0,\infty), f\in B(E), x\in E,
\end{equation}
defines a semigroup on $B(E)$, the set of bounded measurable functions
$f:E\to\rz$ equipped with the supremum norm $\|f\|:=\sup_{x\in E}|f(x)|$.
Note that (\ref{semigroup}) is also well defined for some unbounded
functions, for example for all polynomials $f:E\to\rz$. The
semigroup $(T_t)_{t\ge 0}$ on $B(E)$ is clearly conservative,
since $T_t1=1$ for all $t\in [0,\infty)$. We have
$\|T_tf\|=\sup_{x\in E}|\me(f(x^{e^{-t}}\eta_t))|\le\sup_{x\in E}
\me(|f(x^{e^{-t}}\eta_t)|)\le \|f\|$ for all $t\in [0,\infty)$ and all $f\in B(E)$.
Thus, $\|T_t\|\le 1$ for all $t\in [0,\infty)$, so the semigroup
$(T_t)_{t\ge 0}$ is contracting. Moreover, $(T_t)_{t\ge 0}$ is obviously
positive meaning that each operator $T_t$ maps nonnegative functions
(in $B(E)$) to nonnegative functions.

Let $\widehat{C}(E)\subseteq B(E)$ denote the Banach space of
continuous functions $f:E\to\rz$ vanishing at infinity. Using the
dominated convergence theorem it is easily seen (see Lemma
\ref{applemma} in the appendix) that $T_t\widehat{C}(E)\subseteq
\widehat{C}(E)$ for all $t\in [0,\infty)$. With some more effort
(see again Lemma \ref{applemma}) it can be shown by exploiting the
theorem of Heine that, for all $f\in\widehat{C}(E)$, $T_tf(x)\to
f(x)$ as $t\to 0$ uniformly for all $x\in E$. Therefore,
$(T_t)_{t\ge 0}$ is strongly continuous on
$\widehat{C}(E)$, thus a Feller
%\todo{Feller? $T_tf(x)\to f(x)$ as $t\to 0$ for all $f\in C_0$, $x\in E$?}
semigroup on $\widehat{C}(E)$. Hence (see, for example,
\cite[p.~169, Theorem 2.7]{ethierkurtz}) there exists a Markov
process $X=(X_t)_{t\ge 0}$ corresponding to $(T_t)_{t\ge 0}$ with
initial distribution $\pr(X_0=1)=1$ and sample paths in the space
$D_E[0,\infty)$ of right continuous functions $x:[0,\infty)\to E$
with left limits equipped with the Skorohod topology.
%Hence, by \cite[p.~157, Theorem 1.1]{ethierkurtz} there exists a Markov
%process $X=(X_t)_{t\ge 0}$ with state space $E$, initial distribution
%$\pr(X_0=1)=1$ and transition mechanism
Note that $\me(f(X_{s+t})\,|\,X_u,u\le s)=T_tf(X_s)$ for all $f\in B(E)$ and
all $s,t\in [0,\infty)$ and that
$$
\pr(X_{s+t}\in B\,|\,X_u,u\le s)\ =\ p(t,X_s,B),
\qquad s,t\in [0,\infty), B\in{\cal B}(E).
$$
From
$\me(f(X_t))=\me(f(X_t)\,|\,X_0)=T_tf(1)=\me(f(\eta_t))$, $f\in B(E)$,
$t\in [0,\infty)$, we conclude
that $X_t$ has the same distribution as $\eta_t$, so $X_t$ is
Mittag--Leffler distributed with parameter $e^{-t}$. We therefore
call $X$ the {\em Mittag--Leffler process}.

Clearly, $X_t\to X_\infty$ in distribution as $t\to\infty$, where
$X_\infty$ is standard exponentially distributed. Thus, the stationary
distribution of $X$ is the standard exponential distribution.
\begin{remark}
   The Chapman--Kolmogoroff property holds whenever the random
   variable $\eta_t$ introduced at the beginning of the construction in
   this subsection has
   moments of the form $\me(\eta_t^m)=h(m)/h(me^{-t})$ for some given function
   $h:[0,\infty)\to (0,\infty)$. We have carried out the
   construction for $h(x):=\Gamma(1+x)$
   leading to the Mittag--Leffler process. More generally, one may
   use other functions $h$, for example $h(x):=\Gamma(\beta+x)$
   for some constant $\beta$, leading to a construction of a
   wider class of Markov processes $X=(X_t)_{t\ge 0}$.
\end{remark}
\subsubsection{Further properties of the Mittag--Leffler process}
In this subsection we derive some further properties of the Mittag--Leffler process $X$.
The following lemma provides a formula for the moments of the finite-dimensional
distributions of $X$.
\begin{lemma}[Moments of the finite-dimensional distributions of ${\mathbf X}$]
\label{xmoments}
   Let $k\in\nz$, $0=t_0\le t_1<t_2<\cdots<t_k$ and $m_1,\ldots,m_k\in [0,\infty)$.
   For $j\in\{0,\ldots,k\}$ define $x_j:=x_j(k):=\sum_{i=j+1}^k m_i
   e^{-(t_i-t_j)}$. Note that $x_k=0$ and $x_0=\sum_{i=1}^k m_ie^{-t_i}$.
   Then
   \begin{equation} \label{jointmoments}
      \me(X_{t_1}^{m_1}\cdots X_{t_k}^{m_k})
      \ =\ \prod_{j=1}^k \frac{\Gamma(1+x_j+m_j)}{\Gamma(1+x_{j-1})}
   \end{equation}
   and the entire moments $\me(X_{t_1}^{m_1}\cdots X_{t_k}^{m_k})$,
   $m_1,\ldots,m_k\in\nz_0$, uniquely determine the distribution of
   $(X_{t_1},\ldots,X_{t_k})$. In particular,
   $\me(X_t^m)=\Gamma(1+m)/\Gamma(1+me^{-t})$, $m\in\nz_0$, $t\in
   [0,\infty)$, and, hence, $\me(X_t)=1/\Gamma(1+e^{-t})$ and
   ${\rm
   Var}(X_t)=\me(X_t^2)-(\me(X_t))^2=2/\Gamma(1+2e^{-t})-1/(\Gamma(1+e^{-t}))^2$,
   $t\in [0,\infty)$.
\end{lemma}
\begin{proof}
   Induction on $k$. Clearly, (\ref{jointmoments}) holds for $k=1$, since
   $X_{t_1}$ is Mittag--Leffler distributed with parameter $e^{-t_1}$.
   The induction step from $k-1$ to $k$ works as follows. We have
   $$
   \me(X_{t_1}^{m_1}\cdots X_{t_k}^{m_k})
   \ =\ \me(\me(X_{t_1}^{m_1}\cdots X_{t_k}^{m_k}\,|\,X_{t_1},\ldots,X_{t_{k-1}}))
   \ =\ \me(X_{t_1}^{m_1}\cdots X_{t_{k-1}}^{m_{k-1}}\me(X_{t_k}^{m_k}\,|\,X_{t_{k-1}})).
   $$
   Define $f_m(x):=x^m$ for convenience. Using the formula (\ref{semigroup})
   for the semigroup operator $T_t$, the last conditional expectation is
   given by
   $$
   \me(X_{t_k}^{m_k}\,|\,X_{t_{k-1}})
   \ =\ (T_{t_k-t_{k-1}}f_{m_k})(X_{t_{k-1}})
   \ =\ \frac{\Gamma(1+m_k)}{\Gamma(1+m_ke^{-(t_k-t_{k-1})})}
   X_{t_{k-1}}^{m_ke^{-(t_k-t_{k-1})}}.
   $$
   We therefore obtain
   \begin{eqnarray*}
      \me(X_{t_1}^{m_1}\cdots X_{t_k}^{m_k})
      & = & \frac{\Gamma(1+m_k)}{\Gamma(1+m_ke^{-(t_k-t_{k-1})})}
            \me(X_{t_1}^{m_1}\cdots X_{t_{k-2}}^{m_{k-2}}X_{t_{k-1}}^{m_{k-1}+m_ke^{-(t_k-t_{k-1})}})\\
      & = & \frac{\Gamma(1+x_k+m_k)}{\Gamma(1+x_{k-1})}
            \me(X_{t_1}^{\tilde{m}_1}\cdots X_{t_{k-1}}^{\tilde{m}_{k-1}}),
   \end{eqnarray*}
   where $\tilde{m}_j:=m_j$ for $1\le j\le k-2$ and
   $\tilde{m}_{k-1}:=m_{k-1}+m_ke^{-(t_k-t_{k-1})}$. By induction,
   $$
   \me(X_{t_1}^{\tilde{m}_1}\cdots X_{t_{k-1}}^{\tilde{m}_{k-1}})
   \ =\ \prod_{j=1}^{k-1}\frac{\Gamma(1+y_j+\tilde{m}_j)}{\Gamma(1+y_{j-1})},
   $$
   where $y_j:=\sum_{i=j+1}^{k-1} \tilde{m}_i e^{-(t_i-t_j)}$ for all
   $j\in\{0,\ldots,k-1\}$. The result follows since, for $1\le j\le k-2$,
   $$
   y_j
   \ =\ \sum_{i=j+1}^{k-2} m_ie^{-(t_i-t_j)} + (m_{k-1}+m_ke^{-(t_k-t_{k-1})})e^{-(t_{k-1}-t_j)}
   \ =\ \sum_{i=j+1}^k m_i e^{-(t_i-t_j)}
   \ =\ x_j
   $$
   and $y_{k-1}=0$ and, hence,
   $y_{k-1}+\tilde{m}_{k-1}=\tilde{m}_{k-1}=
   m_ke^{-(t_k-t_{k-1})}+m_{k-1}=x_{k-1}+m_{k-1}$.\hfill$\Box$
\end{proof}
\begin{remark}
   The mean $\me(X_t)=1/\Gamma(1+e^{-t})$ is increasing
   for $t<t_0$ and decreasing for $t>t_0$, where
   $t_0\approx 0.772987$ is the unique solution of the equation
   $\Psi(1+e^{-t_0})=0$ and $\Psi:=\Gamma'/\Gamma$ denotes the logarithmic
   derivative of the gamma function. The process $X$ is therefore neither
   a process with non-increasing paths nor a process with non-decreasing
   paths. In particular, we are not in the context of \cite{haasmiermont},
   where essentially all considered processes have non-increasing paths.
\end{remark}
\begin{corollary}
   The Mittag--Leffler process $X=(X_t)_{t\ge 0}$ is continuous in
   probability, i.e. $X_s\to X_t$ in probability as $s\to t$ for
   every $t\in [0,\infty)$.
\end{corollary}
\begin{proof}
   By Lemma \ref{xmoments}, for all $s,t\in [0,\infty)$,
   $$
   \me(X_s^2)\ =\ \frac{\Gamma(3)}{\Gamma(1+2e^{-s})}
   \ \to\ \frac{\Gamma(3)}{\Gamma(1+2e^{-t})}
   \ =\ \me(X_t^2),\qquad s\to t,
   $$
   and
   $$
   \me(X_sX_t)
   \ =\ \frac{\Gamma(2+e^{-|t-s|})}{\Gamma(1+e^{-s}+e^{-t})}
   \frac{\Gamma(2)}{\Gamma(1+e^{-|t-s|})}
   \ \to\ \frac{\Gamma(3)}{\Gamma(1+2e^{-t})}
   \ =\ \me(X_t^2),\qquad s\to t.
   $$
%   it follows that all moments
%   $\me(X_{t_1}^{m_1}\cdots X_{t_k}^{m_k})$ are
%   continuous in $(t_1,\ldots,t_k)$. In particular
%   $\me(X_sX_t)\to\me(X_t^2)$ and $\me(X_s^2)\to\me(X_t^2)$
%   as $s\to t$.
   It follows that $\me((X_s-X_t)^2)=
   \me(X_s^2)-2\me(X_sX_t)+\me(X_t^2)\to 0$ as $s\to t$.
   Thus, for all $\varepsilon>0$,
   $\pr(|X_s-X_t|\ge\varepsilon)\le\me((X_s-X_t)^2)/\varepsilon^2\to 0$
   as $s\to t$.\hfill$\Box$
\end{proof}
\begin{remark}
   Note that (the Mittag--Leffler distributed random variable) $X_t$ is
   not infinitely divisible.
   Moreover, $X$ does not have independent increments. In particular,
   $X$ is not a L\'evy process. The process $(\log X_t)_{t\ge 0}$ is
   as well not a L\'evy process, since this process does not have
   independent increments either. This can be also seen as follows. The
   Fourier transform
   $\phi_t(x):=\me(e^{ix\log X_t})
   =\me(X_t^{ix})
   =\Gamma(1+ix)/\Gamma(1+ixe^{-t})$, $x\in\rz$,
   of $\log X_t$ is not the $t$-th power of $\phi_1(x)$.
%   To the best of
%   the authors knowledge the Mittag--Leffler process has not been
%   introduced in the literature. It might be possible that this process
%   is introduced quite implicitly without mentioning its relation to
%   the Mittag--Leffler distribution.

   We leave a possible construction of the
   Mittag--Leffler process via L\'evy processes or subordinators, for example
   as a random time change and/or by taking the absolute value of a certain
   L\'evy process, for future work. For related functionals of this type
   (local time processes, Bessel-type processes) we refer the reader to
%   \todo{$S_\alpha^{-\alpha}$?, see James \cite{james}?}
   James \cite{james} and the references therein.
\end{remark}
We finally provide in this subsection some information on the generator $A$ of
the Mittag--Leffler process $X$, but we will not use the
generator $A$ in our further considerations. Suppose that $f\in B(E)$
is infinitely often differentiable and that $f$ satisfies
$f(y)=\sum_{k=0}^\infty (f^{(k)}(x)/k!)(y-x)^k$ for all
$x,y\in E$. Then
$$
\frac{T_tf(x)-f(x)}{t}
\ =\ \sum_{k=1}^\infty \frac{f^{(k)}(x)}{k!}\frac{\me((x^{e^{-t}}\eta_t-x)^k)}{t}.
$$
Let $\Psi:=\Gamma'/\Gamma$ denote the logarithmic derivative of the gamma
function. Since
\begin{equation} \label{ak}
a_k(x)\ :=\ \lim_{t\searrow 0}\frac{\me((x^{e^{-t}}\eta_t-x)^k)}{t}
\ =\
\left\{
   \begin{array}{ll}
      x\Psi(2)-x\log x %=x(1-\gamma-\log x)
         & \mbox{for $k=1$,}\\
      \displaystyle\frac{(-x)^k}{k-1} & \mbox{for $k\in\nz\setminus\{1\}$,}
   \end{array}
\right.
\end{equation}
the generator $A$ of $X$ satisfies
$$
Af(x)\ =\ \sum_{k=1}^\infty \frac{f^{(k)}(x)}{k!} a_k(x)
$$
with $a_k(x)$ defined via (\ref{ak}).
\subsection{Moment calculations} \label{moment}
\setcounter{theorem}{0}
In this section we provide formulas for certain
moments of the block counting process $N^{(n)}=(N_t^{(n)})_{t\ge 0}$
of the Bolthausen--Sznitman $n$-coalescent. In the following we use
for $x\in (0,\infty)$ and $m\in [0,\infty)$ the notation
$[x]_m:=\Gamma(x+m)/\Gamma(x)$. Note that for $m\in\nz_0$ the
symbol $[x]_m=x(x+1)\cdots(x+m-1)$ coincides with the ascending factorial.
The following lemma provides an explicit formula for the expectation
of $[N_t^{(n)}]_m$.
\begin{lemma} \label{mean}
   Fix $n\in\nz$ and $t\in [0,\infty)$. For the Bolthausen--Sznitman
   coalescent the random variable $N_t^{(n)}$ satisfies for all
   $m\in [0,\infty)$
   $$
   \me([N_t^{(n)}]_m)
   \ =\ \Gamma(m+1)\prod_{j=1}^{n-1}\frac{j+me^{-t}}{j}
   \ =\ \Gamma(m+1){{n-1+me^{-t}}\choose{n-1}}
   \ =\ \frac{\Gamma(m+1)}{\Gamma(1+me^{-t})}[n]_{me^{-t}}.
   $$
   In particular,
   $$
   \me(N_t^{(n)})
   \ =\ \prod_{j=1}^{n-1}\frac{j+e^{-t}}{j}
   \ =\ {{n-1+e^{-t}}\choose{n-1}}
   \ =\ \frac{1}{\Gamma(1+e^{-t})}[n]_{e^{-t}}
   \ =\ \frac{\Gamma(n+e^{-t})}{\Gamma(n)\Gamma(1+e^{-t})}
   $$
   and
   $$
   {\rm Var}(N_t^{(n)})
   \ =\ 2\prod_{j=1}^{n-1}\frac{j+2e^{-t}}{j} - \prod_{j=1}^{n-1}\frac{j+e^{-t}}{j}
   - \bigg(\prod_{j=1}^{n-1}\frac{j+e^{-t}}{j}\bigg)^2.
   $$
\end{lemma}
\begin{proof} (of Lemma \ref{mean})
   Fix $n\in\nz$ and $t\in [0,\infty)$. The formula obviously holds
   for $m=0$. Thus, we can assume that $m\in (0,\infty)$.
   Clearly, $\me([N_t^{(n)}]_m)=\sum_{j=1}^n [j]_m p_{nj}(t)$, where
   $p_{nj}(t):=\pr(N_t^{(n)}=j)$. In the following $s(.,.)$ and $S(.,.)$
   denote the Stirling number of the first and second kind respectively.
   Plugging in
   $$
   p_{nj}(t)
   \ =\ (-1)^{n+j}\frac{\Gamma(j)}{\Gamma(n)}\sum_{k=j}^n
   e^{-(k-1)t} s(n,k)S(k,j)
   $$
   (see \cite[Corollary 1.3]{moehlepitters}, Equation (1.3), corrected by
   an obviously missing sign factor $(-1)^{k+j}$) it follows that
   \begin{eqnarray*}
      \me([N_t^{(n)}]_m)
      & = & \sum_{j=1}^n [j]_m (-1)^{n+j}\frac{\Gamma(j)}{\Gamma(n)}
            \sum_{k=j}^n e^{-(k-1)t} s(n,k) S(k,j)\\
      & = & \frac{(-1)^n}{\Gamma(n)}e^t\sum_{k=1}^n s(n,k) (e^{-t})^k
            \sum_{j=1}^k \Gamma(j+m)(-1)^j S(k,j).
   \end{eqnarray*}
   Since $\Gamma(j+m)(-1)^j=\Gamma(m)[m]_j(-1)^j=
   \Gamma(m)(-m)(-m-1)\cdots(-m-j+1)=\Gamma(m)(-m)_j$,
   where $(x)_j:=x(x-1)\cdots (x-j+1)$,
   the last sum simplifies to
   $\sum_{j=1}^k \Gamma(j+m)(-1)^j S(k,j)
   =\Gamma(m)\sum_{j=1}^k (-m)_j S(k,j)
   =\Gamma(m)(-m)^k$. Thus,
   \begin{eqnarray*}
      &   & \hspace{-10mm}\me([N_t^{(n)}]_m)
      \ = \ \frac{(-1)^n}{\Gamma(n)}e^t\sum_{k=1}^n s(n,k)(e^{-t})^k
            \Gamma(m)(-m)^k\\
      & = & \Gamma(m)\frac{(-1)^n}{\Gamma(n)}e^t\sum_{k=1}^n s(n,k) (-me^{-t})^k
      \ = \ \Gamma(m)\frac{(-1)^n}{\Gamma(n)}e^t(-me^{-t})_n
      \ = \ \frac{\Gamma(m)}{\Gamma(n)}e^t[me^{-t}]_n\\
      & = & \Gamma(m+1)\prod_{j=1}^{n-1}\frac{j+me^{-t}}{j}
      \ = \ \Gamma(m+1){{n-1+me^{-t}}\choose {n-1}}
      \ = \ \frac{\Gamma(m+1)}{\Gamma(1+me^{-t})}
            [n]_{me^{-t}}.
   \end{eqnarray*}
   Choosing $m=1$ the formula for the mean of $N_t^{(n)}$ follows immediately.
   The formula for the variance of $N_t^{(n)}$ follows
   from ${\rm Var}(N_t^{(n)})=\me([N_t^{(n)}]_2)-\me(N_t^{(n)})-
   (\me(N_t^{(n)}))^2$.\hfill$\Box$
\end{proof}
The following result (Lemma \ref{meangeneral}) is a generalization of
Lemma \ref{mean}. It % is somewhat technical, but it
will turn out to be quite useful later in order to verify the main convergence result
(Theorem \ref{main}).
\begin{lemma} \label{meangeneral}
   Let $k\in\nz$, $0=t_0\le t_1<t_2<\cdots<t_k$ and $m_1,\ldots,m_k\in [0,\infty)$.
   For $j\in\{0,\ldots,k\}$ define
   $x_j:=x_j(k):=\sum_{i=j+1}^k m_ie^{-(t_i-t_j)}$. Note that
   $0=x_k\le x_{k-1}\le\cdots\le x_2\le x_1\le x_0=\sum_{i=1}^k m_ie^{-t_i}$. Then
   \begin{equation} \label{momentgeneral}
      \me\bigg(\prod_{j=1}^k [N_{t_j}^{(n)}+x_j]_{m_j}\bigg)
      \ =\ [n]_{x_0}\prod_{j=1}^k \frac{\Gamma(1+x_j+m_j)}{\Gamma(1+x_{j-1})}.
   \end{equation}
\end{lemma}
\begin{proof} (of Lemma \ref{meangeneral})
   Induction on $k$. For $k=1$ the assertion holds by Lemma \ref{mean}.
   The induction step from $k-1$ to $k$ ($\ge 2$) works as follows. We have
   \begin{eqnarray}
      \me\bigg(\prod_{j=1}^k [N_{t_j}^{(n)}+x_j]_{m_j}\bigg)
      & = & \me\bigg(\me\bigg(\prod_{j=1}^k[N_{t_j}^{(n)}+x_j]_{m_j}\bigg| N_{t_1}^{(n)},\ldots,N_{t_{k-1}}^{(n)}\bigg)\bigg)\nonumber\\
      & = & \me\bigg(\prod_{j=1}^{k-1}[N_{t_j}^{(n)}+x_j]_{m_j}
               \me([N_{t_k}^{(n)}]_{m_k}\,|\,N_{t_{k-1}}^{(n)})
            \bigg), \label{local}
   \end{eqnarray}
   since $x_k=0$. The process $N^{(n)}=(N_t^{(n)})_{t\ge 0}$ is time-homogeneous.
   Thus, for all $j\in\{1,\ldots,n\}$,
   $$
   \me([N_{t_k}^{(n)}]_{m_k}\,|\,N_{t_{k-1}}^{(n)}=j)
   \ =\ \me([N_{t_k-t_{k-1}}^{(j)}]_{m_k})
   \ =\ \frac{\Gamma(1+m_k)}{\Gamma(1+m_ke^{-(t_k-t_{k-1})})}[j]_{m_ke^{-(t_k-t_{k-1})}},
   $$
   where the last equality holds by Lemma \ref{mean}. Thus,
   $$
   \me([N_{t_k}^{(n)}]_{m_k}\,|\,N_{t_{k-1}}^{(n)})\ =\ \frac{\Gamma(1+m_k)}{\Gamma(1+m_ke^{-(t_k-t_{k-1})})}
   [N_{t_{k-1}}^{(n)}]_{m_ke^{-(t_k-t_{k-1})}}.
   $$
   Plugging this into (\ref{local}) yields
   \begin{eqnarray*}
      \me\bigg(\prod_{j=1}^k [N_{t_j}^{(n)}+x_j]_{m_j}\bigg)
      & = & \frac{\Gamma(1+m_k)}{\Gamma(1+m_ke^{-(t_k-t_{k-1})})}
            \me\bigg(
               \bigg(\prod_{j=1}^{k-1}[N_{t_j}^{(n)}+x_j]_{m_j}\bigg)
               [N_{t_{k-1}}^{(n)}]_{m_ke^{-(t_k-t_{k-1})}}
            \bigg)\\
      & = & \frac{\Gamma(1+x_k+m_k)}{\Gamma(1+x_{k-1})}
            \me\bigg(
               \prod_{j=1}^{k-1}[N_{t_j}^{(n)}+y_j]_{\tilde{m}_j}
            \bigg),
   \end{eqnarray*}
   where $y_j:=x_j$ and $\tilde{m}_j:=m_j$ for $0\le j\le k-2$, $y_{k-1}:=0$
   and $\tilde{m}_{k-1}:=m_{k-1}+x_{k-1}$. By induction,
   $$
   \me\bigg(
      \prod_{j=1}^{k-1} [N_{t_j}^{(n)}+y_j]_{\tilde{m}_j}
   \bigg)
   \ =\ [n]_{y_0}\prod_{j=1}^{k-1}\frac{\Gamma(1+y_j+\tilde{m}_j)}{\Gamma(1+y_{j-1})}
   \ =\ [n]_{x_0}\prod_{j=1}^{k-1}
        \frac{\Gamma(1+x_j+m_j)}{\Gamma(1+x_{j-1})},
   $$
   and (\ref{momentgeneral}) follows immediately, which completes the
   induction.\hfill$\Box$
\end{proof}
\subsection{Proof of Theorem \ref{main}} \label{proof}
\setcounter{theorem}{0}
   The $\sigma$-algebra generated by $X_t^{(n)}$ coincides with
   the $\sigma$-algebra generated by $N_t^{(n)}$. Thus, the Markov property
   of the block counting process $N^{(n)}$ % =(N_t^{(n)})_{t\ge 0}$
   carries over to the scaled block counting process $X^{(n)}$. % =(X_t^{(n)})_{t\ge 0}$.
   Note however that the process $X^{(n)}$ is time-inhomogeneous whereas
   $N^{(n)}$ is time-homogeneous.

   As a warming up we first verify the convergence of the
   finite-dimensional distributions. Afterwards we turn to the
   convergence in $D_E[0,\infty)$. Since the proof of the
   convergence of the one-dimensional distributions turns out to be less
   technical, we start with a consideration of the one-dimensional
   distributions.

   \vspace{2mm}

   {\bf Step 1.} (Convergence of the one-dimensional distributions)
   Recall that $S(.,.)$ denote the Stirling numbers of the
   second kind. Fix $t\in [0,\infty)$. Applying the formula
   \begin{equation} \label{stirlingspecial}
      x^m\ =\ \sum_{i=0}^m (-1)^{m-i}S(m,i)[x]_i,\qquad m\in\nz_0,
   \end{equation}
   it follows that
   $$ % \label{local1}
      \me((X_t^{(n)})^m)
      \ = \ \frac{1}{n^{me^{-t}}}\me((N_t^{(n)})^m)
      \ = \ \sum_{i=0}^m (-1)^{m-i} S(m,i)\frac{\me([N_t^{(n)}]_i)}{n^{me^{-t}}},
      \quad n\in\nz, m\in\nz_0.
   $$
   By Lemma \ref{mean}, $\me([N_t^{(n)}]_i)
   =(\Gamma(i+1)/\Gamma(1+ie^{-t}))[n]_{ie^{-t}}
   =\me(X_t^i)[n]_{ie^{-t}}$, leading to
   $$
   \me((X_t^{(n)})^m)\ =\ \sum_{i=0}^m (-1)^{m-i} S(m,i)\me(X_t^i)
   \frac{[n]_{ie^{-t}}}{n^{me^{-t}}},\qquad n\in\nz, m\in\nz_0.
   $$
   Letting $n\to\infty$ shows that
      $\lim_{n\to\infty}\me((X_t^{(n)})^m)
      =\me(X_t^m)$ for all $m\in\nz_0$.
   This convergence of moments implies
   (see, for example, \cite[Theorems 30.1 and 30.2]{billingsley}),
   the convergence
   $X_t^{(n)}\to X_t$ in distribution as $n\to\infty$. Thus, the
   convergence of the one-dimensional distributions holds.

   \vspace{2mm}

   {\bf Step 2.} (Convergence of the finite-dimensional distributions)
   Let us now turn to the convergence of the $k$-dimensional
   distributions, $k\in\nz$. Fix $0=t_0\le t_1<t_2<\cdots <t_k<\infty$ and
   $m_1,\ldots,m_k\in [0,\infty)$. For $j\in\{0,\ldots,k\}$ define
   $x_j:=x_j(k):=\sum_{i=j+1}^k m_i e^{-(t_i-t_j)}$.
   Note that $x_k=0$ and that $x_0=\sum_{i=1}^k m_ie^{-t_i}$.
   We have
   \begin{eqnarray*}
      \prod_{j=1}^k \bigg(X_{t_j}^{(n)}+\frac{x_j}{n^{e^{-t_j}}}\bigg)^{m_j}
      & = & \prod_{j=1}^k \frac{(N_{t_j}^{(n)}+x_j)^{m_j}}{n^{me^{-t_j}}}
      \ = \ \frac{1}{n^{x_0}}\prod_{j=1}^k (N_{t_j}^{(n)}+x_j)^{m_j}.
   \end{eqnarray*}
   Applying (\ref{stirlingspecial}) it follows that
   \begin{eqnarray*}
      \prod_{j=1}^k \bigg(X_{t_j}^{(n)}+\frac{x_j}{n^{e^{-t_j}}}\bigg)^{m_j}
      & = & \frac{1}{n^{x_0}}\prod_{j=1}^k \bigg(\sum_{i_j=0}^{m_j} (-1)^{m_j-i_j} S(m_j,i_j) [N_{t_j}^{(n)}+x_j]_{i_j}\bigg)\\
      & = & \frac{1}{n^{x_0}}
           \sum_{i_1\le m_1,\ldots,i_k\le m_k}
           \bigg(\prod_{j=1}^k (-1)^{m_j-i_j}S(m_j,i_j)\bigg)
           \bigg(\prod_{j=1}^k [N_{t_j}^{(n)}+x_j]_{i_j}\bigg).
   \end{eqnarray*}
   Taking expectation yields
   \begin{eqnarray}
      &   & \hspace{-2cm}
      \me\bigg(
         \prod_{j=1}^k
         \bigg(X_{t_j}^{(n)}+\frac{x_j}{n^{e^{-t_j}}}\bigg)^{m_j}
      \bigg)\nonumber\\
      & = & \sum_{i_1\le m_1,\ldots,i_k\le m_k}
            \bigg(\prod_{j=1}^k (-1)^{m_j-i_j}S(m_j,i_j)\bigg)
            \frac{1}{n^{x_0}}
            \me\bigg(\prod_{j=1}^k [N_{t_j}^{(n)}+x_j]_{i_j}\bigg).
            \label{local2}
   \end{eqnarray}
   By Lemma \ref{meangeneral}, the last expectation is $O(n^{\sum_{j=1}^k i_je^{-t_j}})$
   and
   $$
   \me(\prod_{j=1}^k [N_{t_j}^{(n)}+x_j]_{m_j})
   \ =\ [n]_{x_0}\prod_{j=1}^k \frac{\Gamma(1+x_j+m_j)}{\Gamma(1+x_{j-1})}
   \ =\ [n]_{x_0}\me(X_{t_1}^{m_1}\cdots X_{t_k}^{m_k}),
   $$
   where the last equality holds by Eq.~(\ref{jointmoments}) from Lemma
   \ref{xmoments}. Thus, letting $n\to\infty$ in (\ref{local2}) yields
   \begin{equation} \label{local3}
      \lim_{n\to\infty}
      \me\bigg(
         \prod_{j=1}^k\bigg(X_{t_j}^{(n)}+\frac{x_j}{n^{e^{-t_j}}}\bigg)^{m_j}
      \bigg)
      \ =\ \me(X_{t_1}^{m_1}\cdots X_{t_k}^{m_k}).
   \end{equation}
   In order to get rid of the disturbing fractions
   $x_j/n^{e^{-t_j}}$ on the left hand side in (\ref{local3})
   one may use the binomial formula
   $$
   \bigg(X_{t_j}^{(n)}+\frac{x_j}{n^{e^{-t_j}}}\bigg)^{m_j}
   \ =\ \sum_{l_j=0}^{m_j} {{m_j}\choose{l_j}}
    \bigg(\frac{x_j}{n^{e^{-t_j}}}\bigg)^{m_j-l_j}(X_{t_j}^{(n)})^{l_j}
   $$
   and conclude from (\ref{local3}) by induction on
   $m:=m_1+\cdots+m_k\in\nz_0$ that
   \begin{equation} \label{joint}
      \lim_{n\to\infty}\me((X_{t_1}^{(n)})^{m_1}\cdots(X_{t_k}^{(n)})^{m_k})
      \ =\ \me(X_{t_1}^{m_1}\cdots X_{t_k}^{m_k}),
      \qquad m_1,\ldots,m_k\in\nz_0.
   \end{equation}
   This convergence of moments implies (see, for example,
   \cite[Problems 30.5 and 30.6]{billingsley}) the convergence
   $(X_{t_1}^{(n)},\ldots,X_{t_k}^{(n)})\to (X_{t_1},\ldots,X_{t_k})$ in
   distribution as $n\to\infty$. Thus, the convergence of the
   finite-dimensional distributions holds.

   \vspace{2mm}

   {\bf Step 3.} (Preparing the proof of the convergence in $D_E[0,\infty)$)

   Let $M(E)$ denote the set of all measurable functions
   $f:E\to\rz$.
   Define $E_n(s):=\{j/n^{e^{-s}}:j\in\{1,\ldots,n\}\}$ for
   all $n\in\nz$ and all $s\in [0,\infty)$ and
   $$
   T_{s,t}^{(n)}f(x)\ :=\ \me(f(X_{s+t}^{(n)})\,|\,X_s^{(n)}=x),
   \qquad n\in\nz, s,t\in [0,\infty), f\in M(E), x\in E_n(s),
   $$
   Note that $(T_{s,t}^{(n)})_{s,t\ge 0}$ is the semigroup of the
   time-inhomogeneous Markov process $X^{(n)}$. Let us verify
   that, for all $s,t\in [0,\infty)$, all polynomials
   $p:E\to\rz$ and all compact sets $K\subseteq E$,
   \begin{equation} \label{polynomial}
      \lim_{n\to\infty} \sup_{x\in E_n(s)\cap K}
      |T_{s,t}^{(n)}p(x)-T_tp(x)|\ =\ 0.
   \end{equation}
   For $m\in\nz_0$ let $p_m:E\to\rz$ denote the $m$-th monomial defined
   via $p_m(x):=x^m$, $x\in E$. Fix $s,t\in [0,\infty)$ and a compact set
   $K\subseteq E$. For $n\in\nz$, $m\in\nz_0$ and $x\in E_n(s)$
   we have
   $$
   T_{s,t}^{(n)}p_m(x)
   \ =\ \me((X_{s+t}^{(n)})^m\,|\,X_s^{(n)}=x)
   \ =\ \frac{\me((N_{s+t}^{(n)})^m\,|\,N_s^{(n)}=xn^{e^{-s}})}{n^{me^{-(s+t)}}}
   \ =\ \frac{\me((N_t^{(xn^{e^{-s}})})^m)}{n^{me^{-(s+t)}}},
   $$
   where the last equality holds since the block counting process
   $N^{(n)}=(N_t^{(n)})_{t\ge 0}$ is time-homogeneous. By
   (\ref{stirlingspecial}) and Lemma \ref{mean} it follows that
   $$
   T_{s,t}^{(n)}p_m(x)
   \ =\ \sum_{i=0}^m (-1)^{m-i}S(m,i)
        \frac{\me([N_t^{(xn^{e^{-s}})}]_i)}{n^{me^{-(s+t)}}}
   \ =\ \sum_{i=0}^m (-1)^{m-i}S(m,i) \me(X_t^i)
        \frac{[xn^{e^{-s}}]_{ie^{-t}}}{n^{me^{-(s+t)}}}.
   $$
   Since $T_tp_m(x)=\me(p_m(x^{e^{-t}}X_t))=\me(X_t^m)x^{me^{-t}}$ it
   follows that
   $$
   T_{s,t}^{(n)}p_m(x) - T_tp_m(x)
   \ =\ \me(X_t^m)\bigg(\frac{[xn^{e^{-s}}]_{me^{-t}}}{n^{me^{-(s+t)}}}-x^{me^{-t}}\bigg)
        + \sum_{i=0}^{m-1}
        (-1)^{m-i}S(m,i)\me(X_t^i)
        \frac{[xn^{e^{-s}}]_{ie^{-t}}}{n^{me^{-(s+t)}}}.
   $$
   It is straightforward to check that this expression converges uniformly
   for all $x\in E_n(s)\cap K$ (even uniformly for all $x$ in any compact
   subset of $E$) to zero as $n\to\infty$. Therefore, (\ref{polynomial})
   holds for every monomial $p:=p_m$, $m\in\nz_0$, and, by linearity, for all
   polynomials $p:E\to\rz$.

   \vspace{2mm}

   {\bf Step 4.} (Convergence in $D_E[0,\infty)$)
   According to a time-inhomogeneous variant of
   \cite[p.~167, Theorem 2.5]{ethierkurtz} it suffices to verify that for all
   $s,t\in [0,\infty)$ and all $f\in\widehat{C}(E)$,
   \begin{equation} \label{toshow}
      \lim_{n\to\infty}\sup_{x\in E_n(s)}
      |T_{s,t}^{(n)}f(x) - T_tf(x)|\ =\ 0.
   \end{equation}
   Fix $s,t\in [0,\infty)$ and $f\in\widehat{C}(E)$.
   Without loss of generality we may assume that $\|f\|>0$.
   Let $\varepsilon>0$.
   Since $f\in\widehat{C}(E)$ and $T_tf\in\widehat{C}(E)$, there exists
   a constant $x_0=x_0(\varepsilon)\ge 1$ such that $|f(x)|<\varepsilon$
   and $|T_tf(x)|<\varepsilon$
   for all $x>x_0$. Moreover, since $\pr(X_{s+t}\le x_0\,|\,X_s=x)=T_t 1_{(-\infty,x_0]}(x)=
   \pr(x^{e^{-t}}X_t\le x_0)=\pr(X_t\le x_0/x^{e^{-t}})\to \pr(X_t\le 0)=0$
   as $x\to\infty$, we can choose a real constant $L=L(\varepsilon)\ge x_0$
   sufficiently large such that
   $\pr(X_{s+t}\le x_0\,|\,X_s=x)<\varepsilon/\|f\|$ for all $x\ge L$.
   For all $n\in\nz$ and all $x\in E_n(s)$ we have
   \begin{eqnarray*}
      |T_{s,t}^{(n)}f(x)|
      & \le & \me(|f(X_{s+t}^{(n)})|\,|\,X_s^{(n)}=x)\\
      & = & \me(|f(X_{s+t}^{(n)})|\,1_{\{X_{s+t}^{(n)}>x_0\}}\,|\,X_s^{(n)}=x) +
           \me(|f(X_{s+t}^{(n)})|\,1_{\{X_{s+t}^{(n)}\le x_0\}}\,|\,X_s^{(n)}=x)\\
      & \le & \varepsilon + \|f\|\,\pr(X_{s+t}^{(n)}\le x_0\,|\,X_s^{(n)}=x).
   \end{eqnarray*}
   By Step 2, the convergence of the two-dimensional distributions holds.
   In particular, for every $x\ge 1$, $\pr(X_{s+t}^{(n)}\le x_0\,|\,X_s^{(n)}
   =\lfloor xn^{e^{-s}}\rfloor/n^{e^{-s}})$ converges to
   $\pr(X_{s+t}\le x_0\,|\,X_s=x)$ as $n\to\infty$ pointwise for all $x\ge 1$.
   Since the map $x\mapsto\pr(X_{s+t}\le x_0\,|\,X_s=x)=
   T_t1_{[0,x_0]}(x)=\me(1_{[0,x_0]}(x^{e^{-t}}X_t))=\pr(x^{e^{-t}}X_t\le x_0)$, $x\ge 1$,
   is continuous, non-increasing and bounded, this convergence holds even
   uniformly for all $x\ge 1$. [The proof of this uniform convergence
   works essentially the same as the proof that pointwise convergence
   of distribution functions holds even uniform, if the limiting
   distribution function is continuous.]
   Thus, there exists $n_0=n_0(\varepsilon)\in\nz$
   such that $\pr(X_{s+t}^{(n)}\le x_0\,|\,X_s^{(n)}=x)\le
   \pr(X_{s+t}\le x_0\,|\,X_s=x)+\varepsilon/\|f\|$ for all $n>n_0$ and all
   $x\in E_n(s)\cap [1,\infty)$.
   For all $n\in\nz$ with $n>n_0$ and all $x\in E_n(s)\cap [L,\infty)$
   it follows that
   \begin{eqnarray*}
      |T_{s,t}^{(n)}f(x)|
      & \le & \varepsilon + \|f\|\,\bigg(\pr(X_{s+t}\le x_0\,|\,X_s=x)+\frac{\varepsilon}{\|f\|}\bigg)\\
      & = & 2\varepsilon + \|f\|\,\pr(X_{s+t}\le x_0\,|\,X_s=x)
      \ \le\ 3\varepsilon.
   \end{eqnarray*}
   Thus, for all $n>n_0$,
   \begin{eqnarray*}
      \sup_{x\in E_n(s)\cap [L,\infty)}|T_{s,t}^{(n)}f(x)-T_tf(x)|
      & \le & \sup_{x\in E_n(s)\cap [L,\infty)}|T_{s,t}^{(n)}f(x)|
      + \sup_{x\in E_n(s)\cap [L,\infty)} |T_tf(x)|\\
      & \le & 3\varepsilon + \varepsilon
      \ =\ 4\varepsilon.
   \end{eqnarray*}
   Thus it is shown that
   $$
   \lim_{n\to\infty} \sup_{x\in E_n(s)\cap [L,\infty)}|T_{s,t}^{(n)}f(x)-T_tf(x)|\ =\ 0.
   $$
   Defining $K:=[0,L]$ it remains to verify that
   \begin{equation} \label{compact}
      \lim_{n\to\infty} \sup_{x\in E_n(s)\cap K}
      |T_{s,t}^{(n)}f(x)-T_tf(x)|\ =\ 0.
   \end{equation}
   By the Tschebyscheff-Markov inequality, for all $y>0$ and all
   $x\in E_n(s)\cap K$,
   $$
   T_t1_{(y,\infty)}(x)
   \ =\
   \pr(X_{s+t}>y\,|\,X_s=x)
   \ \le\ \frac{1}{y}\me(X_{s+t}\,|\,X_s=x)
   \ =\ \frac{1}{y} \me(X_t)x^{e^{-t}}
   \ \le\ \frac{1}{y}\me(X_t)L^{e^{-t}}.
   $$
   Moreover, making again use of the Tschebyscheff-Markov inequality
   and using Lemma \ref{mean}, for all $y>0$ and all $x\in E_n(s)\cap K$,
   \begin{eqnarray*}
      T_{s,t}^{(n)}1_{(y,\infty)}(x)
      & = & \pr(X_{s+t}^{(n)}>y\,|\,X_s^{(n)}=x)
      \ \le \ \frac{1}{y}\me(X_{s+t}^{(n)}\,|\,X_s^{(n)}=x)
      \ = \ \frac{1}{y}\me(X_t)\frac{[xn^{e^{-s}}]_{e^{-t}}}{n^{e^{-(s+t)}}}\\
      & \le & \frac{1}{y}\me(X_t)\frac{[Ln^{e^{-s}}]_{e^{-t}}}{n^{e^{-(s+t)}}}
      \ \sim\ \frac{1}{y}\me(X_t)L^{e^{-t}},\qquad n\to\infty.
   \end{eqnarray*}
   Thus, we can choose a real constant $y_0=y_0(\varepsilon)\ge x_0$
   (which may depend on $s$, $t$ and $L$ but not on $n$) sufficiently large
   such that
   \begin{equation} \label{epsbound}
      T_t 1_{(y_0,\infty)}(x) % \pr(X_{s+t}>y_0\,|\,X_s=x)
      \ \le\ \varepsilon
      \quad\mbox{and}\quad
      T_{s,t}^{(n)}1_{(y_0,\infty)}(x)
      \ \le\ \varepsilon
   \end{equation}
   for all $n\in\nz$ and all $x\in E_n(s)\cap K$. With this choice
   of $y_0$ we are now able to verify (\ref{compact}) as follows. Since
   $|f(y)|<\varepsilon$ for all $y>x_0$ and, hence, for all $y>y_0$, we
   obtain for all $n\in\nz$ and all $x\in E_n(s)$
   $$
   |T_{s,t}^{(n)}f(x)-T_tf(x)|
   \ \le\ 2\varepsilon + |T_{s,t}^{(n)}g(x)-T_tg(x)|,
   $$
   where $g:=f 1_{[0,y_0]}$. By the Weierstrass approximation theorem
   we can approximate the continuous function $f$ uniformly on the compact
   interval $[0,y_0]$ by a polynomial $p$. Hence, there exists a polynomial
   $p$ such that $\|g-h\|<\varepsilon$, where $h:=p1_{[0,y_0]}$. Thus,
   for all $n\in\nz$ and all $x\in E_n(s)$
   $$
   |T_{s,t}^{(n)}f(x)-T_tf(x)|
   \ \le\ 4\varepsilon + |T_{s,t}^{(n)}h(x)-T_th(x)|
   \ \le\ 4\varepsilon + |T_{s,t}^{(n)}p(x)-T_tp(x)|
   + |T_{s,t}^{(n)}r(x)| + |T_tr(x)|,
   $$
   where $r:=p-h=p-p1_{[0,y_0]}=p1_{(y_0,\infty)}$. We have
   already shown in Step 3 that
   $$
   \lim_{n\to\infty}
   \sup_{x\in E_n(s)\cap K}|T_{s,t}^{(n)}p(x)-T_tp(x)|\ =\ 0.
   $$
   Thus it remains to treat $|T_{s,t}^{(n)}r(x)|$ and $|T_tr(x)|$.
   Applying the H\"older inequality and using (\ref{epsbound}) we obtain
   $$
   |T_{s,t}^{(n)}r(x)|
   \ \le\ T_{s,t}^{(n)} p^2(x)\,T_{s,t}^{(n)} 1_{(y_0,\infty)}(x)
   \ \le\ \varepsilon T_{s,t}^{(n)} p^2(x)
   $$
   for all $n\in\nz$ and all $x\in E_n(s)\cap K$.
   Thus it remains to show that $T_{s,t}^{(n)}p^2(x)$ is bounded
   uniformly for all $x\in E_n(s)\cap K$. We have
   $$
   \sup_{x\in E_n(s)\cap K}|T_{s,t}^{(n)}p^2(x)|
   \ \le\
   \sup_{x\in E_n(s)\cap K}|T_{s,t}^{(n)}p^2(x) - T_t p^2(x)|
   + \sup_{x\in K} |T_tp^2(x)|.
   $$
   Since $p^2$ is a polynomial, the first expression converges to zero
   as $n\to\infty$ by Step 3. The last supremum is obviously bounded, since
   $T_tp^2$ is continuous and hence bounded on the compact set $K$, i.e.
   $M:=\sup_{x\in K}|T_tp^2(x)|<\infty$.
   Similarly, by H\"older inequality and (\ref{epsbound}),
   $|T_tr(x)|\le T_tp^2(x) T_t1_{(y_0,\infty)}(x)\le
   \varepsilon T_tp^2(x)\le\varepsilon M$ for all $x\in E_n(s)\cap K$.
   In summary, (\ref{compact}) is established. The proof is complete.
   \hfill$\Box$
\subsection{Appendix}
In this appendix we collect essentially two results. The first
result (Lemma \ref{laplace}) concerns the Laplace exponent of the
subordinator $S$ introduced at the beginning of Section \ref{mittag}.
The second result (Lemma \ref{applemma}) concerns some fundamental
properties of the semigroup $(T_t)_{t\ge 0}$ defined via (\ref{semigroup}).
\begin{lemma} \label{laplace}
   Fix $\alpha\in (0,1)$. The drift-free subordinator $S=(S_t)_{t\ge 0}$
   with killing rate $k:=1/\Gamma(1-\alpha)$ and L\'evy measure
   $\varrho$ with density (\ref{rho}) has Laplace exponent
   (\ref{phi}).
\end{lemma}
\begin{proof}
   By the L\'evy-Khintchine representation, the subordinator $S$ has Laplace
   exponent $\Phi(x)=k+\int_{(0,\infty)} (1-e^{-xu})\,\varrho({\rm d}u)$,
   $x\in [0,\infty)$. Since $\varrho$ has density (\ref{rho}) it follows that
   $$
   \Phi(x)\ =\ k + \frac{1}{\Gamma(1-\alpha)}\int_0^\infty (1-e^{-xu})
   \frac{e^{-u/\alpha}}{(1-e^{-u/\alpha})^{\alpha+1}}\,{\rm d}u.
   $$
   The substitution $y=1-e^{-u/\alpha}$ ($\Rightarrow u=-\alpha\log(1-y)$ and
   ${\rm d}u/{\rm d}y=\alpha/(1-y)$)
   leads to
   \begin{equation} \label{philocal}
      \Phi(x)\ =\ k + \frac{1}{\Gamma(1-\alpha)}\int_0^1
      (1-(1-y)^{\alpha x})\frac{\alpha}{y^{\alpha+1}}\,{\rm d}y.
   \end{equation}
   Partial integration with $u(y):=1-(1-y)^{\alpha x}$ and $v(y):=-y^{-\alpha}$
   turns the last integral into
   \begin{eqnarray*}
      &   & \hspace{-1cm}
            \int_0^1
            (1-(1-y)^{\alpha x})\frac{\alpha}{y^{\alpha+1}}\,{\rm d}y
      \ = \ \big[(1-(1-y)^{\alpha x})(-y^{-\alpha})\big]_0^1
           - \int_0^1 \alpha x(1-y)^{\alpha x-1} (-y^{-\alpha})\,{\rm d}y\\
      & = & -1 + \alpha x \int_0^1 y^{-\alpha}(1-y)^{\alpha x-1}\,{\rm d}y
      \ = \ -1 + \alpha x B(1-\alpha,\alpha x)
      \ = \ -1 + \frac{\Gamma(1-\alpha)\Gamma(1+\alpha x)}{\Gamma(1-\alpha+\alpha x)}.
   \end{eqnarray*}
   Plugging this into (\ref{philocal}) and noting that $k=1/\Gamma(1-\alpha)$
   yields $\Phi(x)=\Gamma(1+\alpha x)/\Gamma(1-\alpha+\alpha x)$, which is
   (\ref{phi}).\hfill$\Box$
\end{proof}
Let $E:=[0,\infty)$ and let $\widehat{C}(E)$ denote the set of
continuous functions $f:E\to\rz$ vanishing at infinity. The following
result is well known, we nevertheless mention it since it will turn out
to be useful to verify fundamental properties of the semigroup
$(T_t)_{t\ge 0}$ defined via (\ref{semigroup}).
\begin{lemma} \label{uniformcont}
   Every $f\in\widehat{C}(E)$ is uniformly continuous on $E$.
\end{lemma}
\begin{proof}
   Let $\varepsilon>0$. Since $f$ vanishes at infinity, there
   exists $x_0\in [0,\infty)$ such that $|f(x)|<\varepsilon/2$ for
   all $x\in [x_0,\infty)$. By the theorem of Heine, $f$
   is uniformly continuous on $[0,x_0+1]$.
   Thus, there exists $\delta=\delta(\varepsilon)\in (0,1)$ such
   that $|f(x)-f(y)|<\varepsilon$ for all $x,y\in [0,x_0+1]$ with
   $|x-y|<\delta$. If $|x-y|<\delta$ but $x>x_0+1$ or $y>x_0+1$,
   then $x\ge x_0$ and $y\ge x_0$ and hence $|f(x)-f(y)|\le
   |f(x)|+|f(y)|<\varepsilon/2+\varepsilon/2=\varepsilon$.
   Thus, $|f(x)-f(y)|<\varepsilon$ for all $x,y\in E$ with
   $|x-y|<\delta$.\hfill$\Box$
\end{proof}
\begin{lemma} \label{applemma}
   For every $t\in [0,\infty)$ the operator $T_t$ defined via
   (\ref{semigroup}) satisfies $T_t\widehat{C}(E)\subseteq\widehat{C}(E)$. Moreover,
   for every $f\in\widehat{C}(E)$, $\lim_{t\to 0}T_tf(x)=f(x)$ uniformly for
   all $x\in E$, so $(T_t)_{t\ge 0}$ is a strongly continuous semigroup on
   $\widehat{C}(E)$.
\end{lemma}
\begin{proof}
   For $t\in [0,\infty)$, $f\in\widehat{C}(E)$ and $x\in E$ we have
   $$
   (T_tf)(x)\ =\ \int_E f(x^{e^{-t}}y)\pr_{\eta_t}({\rm d}y)
   \ \to\ 0,\qquad x\to\infty,
   $$
   by dominated convergence and, similarly,
   $$
   (T_tf)(x)-(T_tf)(x_0)
   \ =\ \int_E (f(x^{e^{-t}}y)-f(x_0^{e^{-t}}y))\,\pr_{\eta_t}({\rm d}y)
   \ \to\ 0,\qquad x\to x_0,
   $$
   again by dominated convergence. Thus $T_t\widehat{C}(E)\subseteq\widehat{C}(E)$
   for all $t\in [0,\infty)$.

   In order to prove the second statement fix $f\in\widehat{C}(E)$.
   Note that $f$ is bounded, i.e. $\|f\|:=\sup_{x\in E}|f(x)|<\infty$.
   For $\alpha\in [0,1]$ let $Z_\alpha$ denote
   a random variable being Mittag--Leffler distributed with parameter
   $\alpha$. Since $T_tf(x)=\me(f(x^{e^{-t}}\eta_t))$, where $\eta_t$ is
   Mittag--Leffler distributed with parameter $\alpha:=e^{-t}$, we
   have to verify that $\lim_{\alpha\to 1}\me(f(x^\alpha Z_\alpha))=f(x)$
   uniformly for all $x\in E$, where without loss of generality we can
   assume that $\alpha\in [1/2,1]$.

   Fix $\varepsilon>0$. Since $f$ vanishes at infinity, there exists a
   constant $x_0\in [1,\infty)$ such that $|f(x)|<\varepsilon$ for all $x\ge x_0$.
   Define $K:=4x_0^2$ ($\ge x_0\ge 1$). In the following the uniform
   convergence $\lim_{\alpha\to 1}\me(f(x^\alpha Z_\alpha))=f(x)$ is
   verified by distinguishing the two situations $x\in [K,\infty)$ and
   $x\in [0,K]$. For $x\in [K,\infty)$ we essentially exploit the fact
   that $f$ vanishes at infinity. For $x\in [0,K]$ the uniform continuity
   of $f$ (Lemma \ref{uniformcont}) comes into play. Let us start with the
   case $x\in [K,\infty)$.

   For all $x\ge K$ and all $z\ge 1/2$ we have
   $x^\alpha z\ge x^\alpha/2\ge \sqrt{x}/2\ge\sqrt{K}/2= x_0$ and, hence,
   $|f(x^\alpha z)|<\varepsilon$. We therefore obtain uniformly for all
   $x\ge K$
   \begin{eqnarray*}
      \me(f(x^\alpha Z_\alpha))-f(x)
      & \le & \int_E |f(x^\alpha z)-f(x)|\,\pr_{Z_\alpha}({\rm d}z)\\
      & \le & \int_{[1/2,\infty)} \underbrace{|f(x^\alpha z)-f(x)|}_{\le 2\varepsilon}\,\pr_{Z_\alpha}({\rm d}z)
        + \int_{[0,1/2)} \underbrace{|f(x^\alpha z)-f(x)|}_{\le 2\|f\|}\,\pr_{Z_\alpha}({\rm d}z)\\
      & \le & 2\varepsilon + 2\,\|f\|\,\pr(Z_\alpha<1/2)
      \ \to \ 2\varepsilon
   \end{eqnarray*}
   as $\alpha\to 1$,
   since $Z_\alpha\to Z_1\equiv 1$ in distribution as $\alpha\to 1$.

   Assume now that $x\in [0,K]$. By Lemma \ref{uniformcont} the function
   $f$ is uniformly continuous on $E$. Thus, there exists a constant
   $\delta=\delta(\varepsilon)>0$ such that
   $|f(y)-f(x)|<\varepsilon$ for all $x,y\in E$ with $|y-x|<\delta$.
   Since $x^\alpha$ converges to $x$ as $\alpha\to 1$ uniformly on $[0,K]$
   we can
   choose $\alpha_0=\alpha_0(\delta)=\alpha_0(\varepsilon)<1$ sufficiently
   close to $1$ such that $|x^\alpha-x|<
   \delta/2$ for all $\alpha\in[\alpha_0,1]$ and all $x\in [0,K]$. For all
   $\alpha\in[\alpha_0,1]$, $x\in [0,K]$ and all
   $z\in E$ with $|z-1|<\gamma:=\delta/(2K)$ we have
   \begin{eqnarray*}
      |x^\alpha z-x|
      & \le & |x^\alpha z-x^\alpha|+|x^\alpha-x|
      \ = \ x^\alpha |z-1| + |x^\alpha-x| \\
      & < & K^\alpha|z-1| + \frac{\delta}{2}
   \ \le\ K|z-1| + \frac{\delta}{2}
   \ <\ K\gamma + \frac{\delta}{2}
   \ =\ \delta,
   \end{eqnarray*}
   and, hence, $|f(x^\alpha z)-f(x)|<\varepsilon$. For all
   $\alpha\in[\alpha_0,1]$ and all $x\in [0,K]$ it follows that
   \begin{eqnarray*}
      &   & \hspace{-2cm}|\me(f(x^\alpha Z_\alpha))-f(x)|
      \ \le\ \int_E |f(x^\alpha z) - f(x)|\,\pr_{Z_\alpha}({\rm d}z)\\
      & = & \int_{\{|z-1|<\gamma\}} \underbrace{|f(x^\alpha z)-f(x)|}_{\le\varepsilon}\,\pr_{Z_\alpha}({\rm d}z)
           + \int_{\{|z-1|\ge \gamma\}} \underbrace{|f(x^\alpha z)-f(x)|}_{\le 2\|f\|}\,\pr_{Z_\alpha}({\rm d}z)\\
      & \le & \varepsilon + 2\,\|f\|\,\pr(|Z_\alpha-1|\ge\gamma)
      \ \to\ \varepsilon
   \end{eqnarray*}
   as $\alpha\to 1$, since $Z_\alpha\to Z_1\equiv 1$ in probability as
   $\alpha\to 1$. In summary it is shown
   that $\lim_{\alpha\to 1}\me(f(x^\alpha Z_\alpha))=f(x)$ uniformly for
   all $x\in E$. Thus,
   $\lim_{t\to 0}T_tf(x)=f(x)$ uniformly for all $x\in E$.\hfill$\Box$
\end{proof}
%\begin{acknowledgement}
%   ...
%\end{acknowledgement}
%
%\vfill\eject
%

\end{document}